\documentclass[11pt]{amsart}
\usepackage{amsmath,amssymb}

\begin{document}

\newtheorem{thm}{Theorem}[section]
\newtheorem{theorem}{Theorem}[section]
\newtheorem{lem}[thm]{Lemma}
\newtheorem{lemma}[thm]{Lemma}
\newtheorem{prop}[thm]{Proposition}
\newtheorem{proposition}[thm]{Proposition}
\newtheorem{cor}[thm]{Corollary}
\newtheorem{defn}[thm]{Definition}
\newtheorem*{remark}{Remark}
\newtheorem{conj}[thm]{Conjecture}

\numberwithin{equation}{section}

\newcommand{\Z}{{\mathbb Z}} 
\newcommand{\Q}{{\mathbb Q}}
\newcommand{\R}{{\mathbb R}}
\newcommand{\C}{{\mathbb C}}
\newcommand{\N}{{\mathbb N}}
\newcommand{\FF}{{\mathbb F}}
\newcommand{\fq}{\mathbb{F}_q}

\newcommand{\rmk}[1]{\footnote{{\bf Comment:} #1}}

\renewcommand{\mod}{\;\operatorname{mod}}
\newcommand{\ord}{\operatorname{ord}}
\newcommand{\TT}{\mathbb{T}}
\renewcommand{\i}{{\mathrm{i}}}
\renewcommand{\d}{{\mathrm{d}}}
\renewcommand{\^}{\widehat}
\newcommand{\HH}{\mathbb H}
\newcommand{\Vol}{\operatorname{vol}}
\newcommand{\area}{\operatorname{area}}
\newcommand{\tr}{\operatorname{tr}}
\newcommand{\norm}{\mathcal N} 
\newcommand{\intinf}{\int_{-\infty}^\infty}
\newcommand{\ave}[1]{\left\langle#1\right\rangle} 
\newcommand{\Var}{\operatorname{Var}}
\newcommand{\Prob}{\operatorname{Prob}}
\newcommand{\sym}{\operatorname{Sym}}
\newcommand{\disc}{\operatorname{disc}}
\newcommand{\CA}{{\mathcal C}_A}
\newcommand{\cond}{\operatorname{cond}} 
\newcommand{\lcm}{\operatorname{lcm}}
\newcommand{\Kl}{\operatorname{Kl}} 
\newcommand{\leg}[2]{\left( \frac{#1}{#2} \right)}  
\newcommand{\Li}{\operatorname{Li}}

\newcommand{\sumstar}{\sideset \and^{*} \to \sum}

\newcommand{\LL}{\mathcal L} 
\newcommand{\sumf}{\sum^\flat}
\newcommand{\Hgev}{\mathcal H_{2g+2,q}}
\newcommand{\USp}{\operatorname{USp}}
\newcommand{\conv}{*}
\newcommand{\dist} {\operatorname{dist}}
\newcommand{\CF}{c_0} 
\newcommand{\kerp}{\mathcal K}

\title[variance of the number of prime polynomials]{The variance of the number of prime polynomials in short intervals and in residue classes}
\author{J.P. Keating and Z. Rudnick}

\address{School of Mathematics, University of Bristol, Bristol BS8 1TW, UK}
\email{j.p.keating@bristol.ac.uk}

\address{Raymond and Beverly Sackler School of Mathematical Sciences,
Tel Aviv University, Tel Aviv 69978, Israel}
\email{rudnick@post.tau.ac.il}
\date{\today}
\thanks{JPK was supported by a grant from the Leverhulme Trust and by the Air Force Office of Scientific Research, Air Force Material Command, USAF, under grant number FA8655-10-1-3088. The U.S. Government is authorized to reproduce and distribute reprints for Governmental purpose notwithstanding any copyright notation thereon.  ZR was supported by   the Israel Science Foundation (grant No.
1083/10).}

\begin{abstract}
We resolve a function field version of two conjectures  concerning
the variance of the number of primes in short intervals (Goldston
and Montgomery) and in arithmetic progressions (Hooley).
 A crucial ingredient in our work are recent equidistribution results of N. Katz.
\end{abstract}
\maketitle

\numberwithin{equation}{section}

\section{Introduction}

In this note we study a function field version of two outstanding
problems in classical Prime Number Theory, concerning the variance
of the number of primes in short intervals and in arithmetic
progressions.

\subsection{Problem 1: Primes in short intervals}
The Prime Number Theorem (PNT) asserts that the number $\pi(x)$ of
primes up to $x$ is asymptotically $\Li(x) = \int_2^x \frac{dt}{\log
  t}$. Equivalently, defining the von Mangoldt function as
$\Lambda(n)=\log p$ if $n=p^k$ is a prime power, and $0$ otherwise,
then PNT is equivalent to the assertion that
\begin{equation}
\psi(x):=\sum_{n\leq x}\Lambda(n) \sim x \quad \mbox{ as }x\to
\infty \;.
\end{equation}

To study  the distribution of primes in short intervals, we define
for  $1\leq H\leq x$,
\begin{equation}
 \psi(x;H):=   \sum_{n\in[x-\frac H2,x+\frac H2]} \Lambda(n) \;.
\end{equation}
The Riemann Hypothesis guarantees an asymptotic formula
$\psi(X;H)\sim H$ as long as $H>X^{\frac 12 +o(1)}$. To understand
the behavior in shorter intervals, Goldston and Montgomery
\cite{Goldston-Montgomery} studied the variance of $\psi(x;H)$ and
showed conditionally that for $X^\delta<H<X^{1-\delta}$,
\begin{equation}\label{MS result}
\frac 1X\int_2^X \left|\psi(x;H)- H\right|^2 dx \sim H(\log X-\log
H)\;,
\end{equation}
assuming the Riemann Hypothesis and  the ("strong") pair correlation
conjecture.  Furthermore, they showed that under RH \eqref{MS
result} and the strong pair correlation conjecture are in fact
equivalent. At this time \eqref{MS result} is still open.

\subsection{Problem 2: Primes in arithmetic progressions}
The Prime Number Theorem for arithmetic progression states that for
a modulus $Q$ and $A$ coprime to $Q$, the number of primes $p\leq X$
with $p=A\bmod Q$ is asymptotically $\pi(x)/\phi(Q)$, where $\pi(X)$
is the number of primes up to $X$ and $\phi(Q)$ is the Euler totient
function, giving the number of reduced residues modulo $Q$.
Equivalently, if
\begin{equation}
  \psi(X;Q,A) := \sum_{\substack{n\leq X\\ n=A\bmod Q}} \Lambda(n)
\end{equation}
then PNT for arithmetic progressions states that for a fixed modulus
$Q$,
\begin{equation}\label{PNT for arith prog}
   \psi(X;Q,A)\sim \frac{X}{\phi(Q)},\quad \mbox{ as }X\to \infty
   \;.
\end{equation}

In most arithmetic applications it is crucial to allow the modulus
to grow with $X$. Thus the remainder term in \eqref{PNT for arith
prog} is of the essence. For very large moduli $Q>X$, there can be
at most one prime in the arithmetic progression $P=A \mod Q$ so that
the interesting range is $Q<X$. Assuming the Generalized Riemann
Hypothesis (GRH) gives \eqref{PNT for arith prog} for
$Q<X^{1/2-o(1)}$.

The fluctuations of $\psi(X;Q,A)$ have been studied over several
decades, notably allowing also averaging over the modulus $Q$. Thus
define
\begin{equation}
  G(X,Q)=\sum_{\substack{A\bmod Q\\ \gcd(A,Q)=1}} \left| \psi(X;Q,A)-\frac X{\phi(Q)}
  \right|^2
\end{equation}
and
\begin{equation}
  H(X,Q ) = \sum_{Q' \leq Q} G(X,Q\, ') \;.
\end{equation}
The study of the sum $H(X,Q)$ has a long history, going under the
name of theorems of Barban-Davenport-Halberstam type. Among other
results is the one due to Montgomery \cite{Montgomery} and Hooley
\cite{HooleyI} asserting that for $X/(\log X)^A<Q<X$ one has
\begin{equation}\label{BDH thm}
  H(X,Q) = QX\log Q -cQX+O(Q^{5/4}X^{3/4} + \frac{X^2}{(\log X)^A})\;,
\end{equation}
for all $A>0$, where
\begin{equation}
   c=\gamma + \log(2\pi) +1  + \sum_p \frac{\log p}{p(p-1)} \;.
\end{equation}
Hooley \cite{HooleyII} showed  that assuming GRH, \eqref{BDH thm}
holds for $X^{1/2+\epsilon}<Q<X$ with remainder $O(X^2/(\log X)^A)$.

The individual variance $G(X,Q)$ is much less understood. Hooley
\cite{HooleyICM} conjectured that under some (unspecified)
conditions,
\begin{equation}\label{Hooley conj}
  G(X,Q) \sim X\log Q \;.
\end{equation}
Friedlander and Goldston \cite{FG} show that in the
range $Q>X$, 
\begin{equation}\label{FG uninteresting}
  G(X,Q)= X\log X - X - \frac{X^2}{\phi(Q)} + O(\frac X{(\log
    X)^A}) + O((\log Q)^3) \;.
\end{equation}
Note that in this range, there is at most one integer $n=A\mod Q$
with $n<X$. They conjecture that \eqref{Hooley conj} holds if
\begin{equation}\label{FG range}
  X^{1/2+\epsilon}<Q<X
\end{equation}
and further conjecture that if $X^{1/2+\epsilon}<Q<X^{1-\epsilon}$
then
\begin{equation}\label{FG conj}
  G(X,Q)= X\log Q -X(\gamma +\log 2\pi +
\sum_{p\mid Q} \frac{\log p}{p-1} )+o(X) \;.
\end{equation}
They show  that both \eqref{Hooley conj} (in the range
$X^{1/2+\epsilon}<Q<X$) and \eqref{FG conj} (in the range
$X^{1/2+\epsilon}<Q<X^{1-\epsilon}$) hold assuming a
Hardy-Littlewood conjecture with small remainders.

For $Q<X^{1/2}$ very little seems to be known. Hooley addresses this
in  paper V of his series of papers on the subject \cite{HooleyV},
which he opens by stating
\begin{quotation}
{\sl An interesting anomaly in the theory of primes is presented by
the situation in which known forms of the prime number theorem for
arithmetic progressions are only valid for (relatively) small values
of the common difference\footnote{Hooley's $k$ corresponds to $Q$
and $x$ to $X$} $k$, whereas the theorems of
Barban-Davenport-Halberstam type discussed \footnote{Here he is
referring to the earlier papers in the series} in I, II, IV are only
fully significant for the (relatively) larger values of $k$. The
most striking illustration of this contrast is perhaps provided by
the conditional theorems at present available on the extended
Riemann hypothesis, the ranges of significance of the prime number
theorem and of the Barban-Montgomery theorem given in II being then,
respectively, $k < x^{1/2-\epsilon}$ and $k > x^{1/2+\epsilon}$.

\dots it is therefore certainly desirable to elicit further forms of
the Barban-Davenport-Halberstam theorem that should be valid for the
smaller values of $k$. }
\end{quotation}

Concerning Conjectures \eqref{Hooley conj} and \eqref{FG conj} for
$G(X,Q)$, Friedlander and Goldston  say \cite[page 315]{FG}
\begin{quotation}
\sl It may well be that these also hold for smaller $Q$, but below
  $X=Q^{1/2}$ we are somewhat skeptical.
\end{quotation}

In this paper we  resolve the function-field versions of
Conjectures~\eqref{MS result} and  \eqref{Hooley conj}, indicating
that \eqref{Hooley conj} should hold all the way down to
$Q>X^\epsilon$.  A crucial ingredient in our work are recent
equidistribution results of Katz \cite{KatzKR, KatzKR2} described in
\S~\ref{secGM}, \S~\ref{secFGH}.

\section{Results for function fields}

Let $\fq$ be a finite field of $q$ elements and $\fq[T]$ the ring of
polynomials with coefficients in $\fq$. Let $\mathcal P_n=\{f\in
\fq[T]:\deg f=n\}$ be the set of polynomials of degree $n$ and
$\mathcal M_n\subset \mathcal P_n$ the subset of monic polynomials.

The von Mangoldt function in this case is defined as
$\Lambda(N)=\deg P$, if $N=cP^k$  with $P$ an irreducible  monic
polynomial, and $c\in \fq^\times$, and $\Lambda(N)=0$ otherwise. The
Prime Polynomial Theorem in this context is the identity
\begin{equation}\label{Explicit formula}
 \sum_{f\in \mathcal M_n}\Lambda(f) = q^n \;.
\end{equation}
\subsection{Short intervals}

For $A\in \mathcal P_n$ of degree $n$, and $h<n$, we define ``short
intervals''
\begin{equation}
I(A;h):=\{f: ||f-A||\leq q^h \} = A+  \mathcal P_{\leq h}\;,
\end{equation}
where the norm of a polynomial $0\neq f\in \fq[T]$ is
\begin{equation}
||f||:=q^{\deg f}
\end{equation}
and
\begin{equation}
  \mathcal P_{\leq h} = \{0\} \cup \bigcup _{0\leq m\leq h}\mathcal P_m
\end{equation}
is the space of polynomials of degree at most $h$ (including $0$).
We have
\begin{equation}
\#I(A;h) = q^{h+1}\;.
\end{equation}

Note: For $h<n$, if $||f-A||\leq q^h$ then $A$ monic if and only if
$f$ is monic. Hence for $A$ monic, $I(A;h)$ consists of only monic
polynomials and all monic $f$'s of degree $n$ are contained in one
of the intervals $I(A;h)$ with $A$ monic.


We define for   $1\leq h< n$ and $A\in \mathcal P_n$,
\begin{equation}
  \nu(A;h) = \sum_{\substack{f\in I(A;h)\\ f(0)\neq 0}} \Lambda(f)
\end{equation}
to be the number of prime powers co-prime to $T$ in the interval
$I(A;h)$, weighted by the degree of the corresponding prime.

We will show in Lemma~\ref{lem:mean value}  that the mean value of
$\nu(A;h)$  when we average over monic $A\in \mathcal M_n$ is
\begin{equation}\label{mean}
\ave{\nu(\bullet;h)}:= \frac 1{q^n} \sum_{A\in \mathcal M_n}
\nu(A;h) =q^{h+1}(1-\frac{1}{q^n})\;.
\end{equation}
 Our goal is to compute the variance
 $$\Var\nu(\bullet;h) = \frac 1{q^n}\sum_{A\in \mathcal M_n} |\nu(A;h)-\ave{\nu(\bullet;h)}|^2$$
 in the limit $q\to\infty$.

\begin{theorem}\label{thm:GMff}
Let $h<n-3$. Then
\begin{equation}\label{GMff}
\lim_{q\to \infty} \frac 1{  q^{h+1}} \Var(\nu(\bullet;h))= n-h-2\;.
\end{equation}
\end{theorem}

We may compare \eqref{GMff} with \eqref{MS result} if we make the
dictionary
\begin{equation}
X\leftrightarrow q^n, \quad H\leftrightarrow q^{h+1},\quad \log X
\leftrightarrow n, \quad \log H \leftrightarrow h+1\;,
\end{equation}
the conclusion being that Theorem~\ref{thm:GMff} is precisely the
analogue of the conditional result \eqref{MS result} of Goldston and
Montgomery.

\subsection{Arithmetic progressions}
Our second result concerns the analogue of the conjectures of Hooley
\eqref{Hooley conj} and Friedlander-Goldston \eqref{FG conj} and
allows us to make a definite conjecture in that case.

For a  polynomial $Q\in \fq[T]$ of positive degree, and $A\in
\fq[T]$ coprime to $Q$ and any $n>0$, set
\begin{equation}
  \Psi(n;Q,A)=\sum_{\substack{N\in \mathcal M_n\\N=A\bmod Q}} \Lambda(N)
\end{equation}
(the sum over monic polynomials). The Prime Polynomial Theorem in
arithmetic progressions states that as $n\to \infty$,
\begin{equation}
  \Psi(n;Q,A)\sim \frac{q^n}{\Phi(Q)}
\end{equation}
where $\Phi(Q)$ is the Euler totient function for this context,
namely the number of reduced residue classes modulo $Q$. Now set
\begin{equation}
  G(n;Q)=\sum_{\substack{A\mod Q\\ \gcd(A,Q)=1}} \left|
  \Psi(n;Q,A)-\frac{q^n}{\Phi(Q)} \right|^2 \;.
\end{equation}

We wish to show an analogue of Conjecture \eqref{Hooley conj} in the
limit of large finite field size, that is $q\to \infty$.

\begin{theorem}\label{main thm}

i)  Given a finite field $\fq$, let $Q\in \fq[T]$ be a polynomial of
positive degree, and  $1\leq n<\deg Q$. Then
\begin{equation}
G(n;Q) = nq^n - \frac{q^{2n}}{\Phi(Q)} + O(n^2q^{n/2})+O((\deg Q)^2)
\;,
\end{equation}
the implied constant absolute.

 ii)  Fix $n\geq 2$. Given a sequence
of finite fields $\fq$ and square-free polynomials $Q(T)\in \fq[T]$
of positive degree with $n\geq \deg Q-1$, then as  $ q\to \infty$,
\begin{equation}\label{RMT conseq}
    G(n;Q) \sim q^n(\deg Q-1) \;.
\end{equation}
 \end{theorem}
We can compare \eqref{RMT conseq} to  \eqref{Hooley conj} in the
range \eqref{FG range}, if we make the dictionary
\begin{equation}
  Q\leftrightarrow ||Q||=q^{\deg Q}, \quad \log Q \leftrightarrow \deg
  Q, \quad  X\leftrightarrow q^n,\quad \log X\leftrightarrow n \;.
\end{equation}
The result \eqref{FG uninteresting} in the 
range $Q>X$ corresponds to $n<\deg Q$, and the range
$X^{1/2}<Q<X^{}$ of \eqref{FG range} corresponds to $\deg Q<n<2\deg
Q$, so that we recover the function field version of conjecture
\eqref{Hooley conj}. Note that \eqref{RMT conseq} holds for all $n$,
not just that range. Thus Conjecture \eqref{Hooley conj} may well be
valid for all $Q>X^{\epsilon}$.

\section{Background on characters and L-functions}
\label{sec:Background}

We review some standard background concerning Dirichlet L-functions
for the rational function field, see e.g. \cite{Rosen, Weil}.

\subsection{The Prime Polynomial Theorem}
Let $\fq$ be a finite field of $q$ elements and $\fq[T]$ the
polynomials over $\FF$.
The zeta function $Z(u)$ of $\fq[T]$ is
\begin{equation}\label{def Zeta}
Z(u) :=\prod_{P} (1-u^{\deg P})^{-1}
\end{equation}
where the product is over all monic irreducible polynomials in
$\fq[T]$. The product is absolutely convergent for $|u|<1/q$.

By unique factorization into irreducibles in $\fq[T]$, we have for
$|u|<1/q$,
 \begin{equation}\label{form of Zeta}
Z(u) = \frac 1{1-qu}\;.
\end{equation}
Taking the logarithmic derivative of \eqref{def Zeta} and
\eqref{form of Zeta} leads to the ``Explicit formula''
\begin{equation}\label{Explicit formula2}
\Psi(n):=  \sum_{N\in \mathcal M_n}\Lambda(N) = q^n
\end{equation}
from which we immediately deduce the Prime Polynomial Theorem, for
the number $\pi(n)$ of monic irreducible polynomials of degree $n$:
\begin{equation}
  \pi(n) = \frac {q^n}n +O(q^{n/2}) \;.
\end{equation}

\begin{lemma}\label{variants of PNT}
  \begin{equation}
    \sum_{N\in \mathcal M_n} \Lambda(N)^2 = nq^n +O(n^2q^{n/2})
  \end{equation}
where the  implied constant is absolute (independent of $q$ and
$n$).
\end{lemma}

\begin{proof}
We start with the Explicit Formula \eqref{Explicit formula2}
\begin{equation}
  \sum_{d\mid m} d\pi(d) = q^m
\end{equation}
and hence
\begin{equation}
   m\pi(m)\leq  q^m\;.
\end{equation}
Now
\begin{equation}
  q^n=\sum_{d\mid n} d\pi(d) = nq^n + \sum_{\substack{d\mid n\\ d<n}} d\pi(d)
\end{equation}
and hence
\begin{equation}\label{PNT with remainder}
  \pi(n) = \frac {q^n}{n} + O(q^{n/2})\;.
\end{equation}

Likewise
\begin{equation}\label{Eq for Lambdasquared}
   \sum_{N\in \mathcal M_n} \Lambda(N)^2 =\sum_{d\mid n} d^2\pi(d) = n^2\pi(n)
   + \sum_{\substack{d\mid n\\ d<n}} d^2\pi(d)
\end{equation}
with remainder term bounded by
\begin{equation}
  \sum_{\substack{d\mid n\\ d<n}} d^2\pi(d)\leq \sum_{d\leq n/2}
  d^2\pi(d)\leq \sum_{1\leq d\leq n/2} nq^{d/2}
\leq n q\frac{q^{n/2}-1}{q-1}<2nq^{n/2}\;.
\end{equation}
Inserting \eqref{PNT with remainder} into \eqref{Eq for
Lambdasquared} gives the claim.
\end{proof}

\subsection{Dirichlet characters}
For a polynomial $Q(x)\in\fq[T]$  of positive degree, we denote by
$\Phi(Q)$ the order of the group $\left( \fq[T]/(Q) \right)^\times$
of invertible residues modulo $Q$.
A Dirichlet character modulo $Q$ is a homomorphism
$$\chi:\left( \fq[T]/(Q) \right)^\times \to \C^\times$$
that is, after extending $\chi$ to vanish on polynomials which are
not coprime to $Q$, we require $\chi(fg) = \chi(f)\chi(g)$ for all
$f,g\in \fq[T]$, $\chi(1)=1$ and $\chi(f+hQ) = \chi(f)$ for all
$f,h\in \fq[T]$. The number of Dirichlet characters modulo $Q$ is
$\Phi(Q)$.

The orthogonality relations for Dirichlet characters are
\begin{equation}\label{orthogonality relation 1}
  \frac 1{\Phi(Q)} \sum_{\chi \bmod Q} \bar \chi(A) \chi(N)
  = \begin{cases} 1,& N=A\bmod Q \\ 0,& \mbox{ otherwise} \end{cases}
\end{equation}
where the sum is over all Dirichlet characters mod $Q$ and $A$ is
coprime to $Q$, and
\begin{equation}\label{OR2}
  \frac 1{\Phi(Q)}\sum_{A\bmod Q} \chi_1(A)\bar\chi_2(A) =
  \begin{cases}
    1,& \chi_1=\chi_2\\0,&\mbox{ otherwise.}
  \end{cases}
\end{equation}

A Dirichlet character  $\chi$ is ``even''
if $\chi(cF)=\chi(F)$ for $0\neq c\in \fq$. This is in analogy to
the number field case, where a Dirichlet character is called "even"
if $\chi(-1)=+1$, and "odd" if $\chi(-1)=-1$. The number
$\Phi^{ev}(Q)$ of even characters modulo $Q$ is
\begin{equation}
\Phi^{ev}(Q)=\frac 1{q-1} \Phi(Q)\;.
\end{equation}

We require the following orthogonality relations  for even Dirichlet
characters
\begin{lemma}\label{special orthogonality}
Let $\chi_1,\chi_2$ be Dirichlet characters modulo $T^m$, $m>1$.
Suppose $\overline{\chi_1}\chi_2$ is even. Then
\begin{equation}
\frac{1}{ q^{m-1}} \sum_{\substack{B \mod T^m\\
B(0)=1}}\overline{\chi_1}(B)\chi_2(B) = \delta_{\chi_1,\chi_2}\;.
\end{equation}
\end{lemma}
\begin{proof}
We start with the standard orthogonality relation
\begin{equation}\label{standard orthogonality}
\frac{1}{ \Phi(T^m)} \sum_{ B \mod T^m
}\overline{\chi_1}(B)\chi_2(B) = \delta_{\chi_1,\chi_2}\;.
\end{equation}
The only nonzero contributions in the sum are those  $B$ with
$B(0)\neq 0$ (equivalently coprime to $T^m$). We can write each such
$B$ uniquely as $B = cB_1$, with $B_1(0)=1$. Since
$\overline{\chi_1}\chi_2$ is even, we have
\begin{equation}
\overline{\chi_1}\chi_2(cB_1) =  \overline{\chi_1}\chi_2(B_1)
\end{equation}
and hence
\begin{equation}
 \sum_{ B \mod T^m }\overline{\chi_1}(B)\chi_2(B) =(q-1)\sum_{\substack{B \mod T^m\\
B(0)=1}}\overline{\chi_1}(B)\chi_2(B)\;.
\end{equation}
Comparing with \eqref{standard orthogonality} and using
$\Phi(T^m)=(q-1)q^{m-1}$ gives the required result.
\end{proof}

\subsection{Primitive characters}
A character is {\em primitive} if there is no proper divisor $Q'\mid
Q$ so that $\chi(F)=1$ whenever $F$ is coprime to $Q$ and $F=1\mod
Q'$. Denoting by $\Phi_{prim}(Q)$ the number of primitive characters
modulo $Q$, we clearly have $\Phi(Q) =\sum_{D\mid Q} \Phi_{prim}(D)$
and hence by M\"obius inversion,
\begin{equation}
 \Phi_{prim}(Q) = \sum_{D\mid Q} \mu(D) \Phi(\frac QD)
\end{equation}
the sum over all monic polynomials dividing $Q$. Therefore
\begin{equation}
\left| \frac{\Phi_{prim}(Q)}{\Phi(Q)}-1 \right| \leq \frac{2^{\deg
Q}}{q} \;.
\end{equation}
Hence as $q\to \infty$, almost all characters are primitive in the
sense that
\begin{equation}
\frac{\Phi_{prim}(Q)}{\Phi(Q)} = 1 +O(\frac 1q)\;,
\end{equation}
the implied constant depending only on $\deg Q$.

Likewise, the number $\Phi_{prim}^{ev}(Q)$ of primitive even
characters is given by
\begin{equation}
\Phi_{prim}^{ev}(Q)=\sum_{D\mid Q} \mu(D) \Phi^{ev}(\frac QD) =
\frac 1{q-1} \sum_{D\mid Q} \mu(D) \Phi(\frac QD) \;.
\end{equation}
For instance, for $Q(T) = T^m$, $m\geq 2$, we find
\begin{equation}
\Phi_{prim}^{ev}(T^m) = q^{m-2}(q-1)\;.
\end{equation}
The number $\Phi^{prim}_{odd}(Q)$ of {\em odd} primitive characters
is then
\begin{equation}
\Phi_{prim}^{odd}(Q) = \Phi_{prim}(Q)-\Phi_{prim}^{ev}(Q) = (1-\frac
1{q-1}) \Phi_{prim}(Q)
\end{equation}
and hence we find that as $q\to \infty$ with $\deg Q$ fixed, almost
all characters are primitive and odd:
\begin{equation}
\frac{\Phi_{prim}^{odd}(Q)}{\Phi(Q)} = 1 +O(\frac 1q)\;,
\end{equation}
the implied constant depending only on $\deg Q$.

\subsection{L-functions}

The L-function $\mathcal L(u,\chi)$ attached to $\chi$ is defined as
\begin{equation}\label{Def of L}
\mathcal L(u,\chi) = \prod_{P\nmid Q} (1-\chi(P)u^{\deg P})^{-1}
\end{equation}
where the product is over all monic irreducible polynomials in
$\fq[T]$. The product is absolutely convergent for $|u|<1/q$. If
$\chi=\chi_0$ is the trivial character modulo $q$, then
\begin{equation}
\mathcal L(u,\chi_0) = Z(u) \prod_{P\mid Q} (1-u^{\deg P})\;.
\end{equation}

The basic fact about $\mathcal L(u,\chi)$ is that if $Q\in \fq[T]$
is a polynomial of degree $\deg Q\geq 2$, and $\chi\neq \chi_0$  a
nontrivial character mod $Q$, then the L-function
$\mathcal L(u,\chi)$ is  a polynomial in $u$ of degree 
$\deg Q-1$.

 Moreover, if $\chi$ is an ``even'' character
, that is
$\chi(cF)=\chi(F)$ for $0\neq c\in \fq$, then 
there is a "trivial" zero at $u=1$: $\mathcal L(1,\chi)=0$ and hence
\begin{equation}
\mathcal L(u,\chi) = (1-u)P(u,\chi)
\end{equation}
where $P(u,\chi)$ is a polynomial of degree  $\deg Q-2$.

We may factor $\mathcal L(u,\chi)$  in terms of the inverse roots
\begin{equation}
\mathcal L(u,\chi) =\prod_{j=1}^{\deg Q-1}(1-\alpha_j(\chi)u) \;.
\end{equation}
The Riemann Hypothesis, proved by Andre Weil (1948) in general,
is that for each (nonzero) inverse root, either $\alpha_j(\chi)=1$
or
\begin{equation}\label{eqRHWeil}
|\alpha_j(\chi)| = q^{1/2} \;.
\end{equation}


We define
\begin{equation}
\Psi(n,\chi) :=\sum_{\deg f=n} \Lambda(f)\chi(f)
\end{equation}
the sum over monic polynomials of degree $n$. Taking logarithmic
derivative of the L-function gives a formula for $\Psi(n,\chi)$ in
terms of the inverse roots $\alpha_j(\chi)$: If $\chi\neq \chi_0$ is
nontrivial then
\begin{equation}
\Psi(n,\chi) = -\sum_{j=1}^{\deg Q-1} \alpha_j(\chi)^n \;.
\end{equation}
The Riemann Hypothesis \eqref{eqRHWeil} gives for $n>0$
\begin{equation} \label{RHWeil}
| \Psi(n,\chi)|\leq (\deg Q-1)q^{n/2}, \quad \chi\neq \chi_0 \;.
\end{equation}

\subsection{The unitarized Frobenius matrix}
We may state the results in cleaner form if we assume that $\chi$ is
a {\em primitive} character modulo $Q$.

We also define
\begin{equation}
\lambda_\chi:=
\begin{cases}1,&\chi\mbox{ ``even''} \\0,&  \mbox{ otherwise.}\end{cases}
\end{equation}
Then for $Q\in \fq[T]$  a polynomial of   degree $\geq 2$, and
$\chi$ a primitive Dirichlet character modulo $Q$,
$$L^*(u,\chi) := (1-\lambda_\chi u)^{-1}L(u,\chi)$$
is a polynomial of degree 
\begin{equation}\label{def of matrix size N}
N=\deg Q-1-\lambda_\chi
\end{equation}
so that $L^*(u,\chi)=\prod_{j=1}^{N} (1-\alpha_j(\chi)u)$ and
\begin{equation}\label{eq:RHprim}
|\alpha_j|=\sqrt{q}, \qquad \forall j=1,\dots,N\;.
\end{equation}

For a primitive character modulo $Q$, we write the inverse roots as
$\alpha_j=q^{1/2}e^{i\theta_j}$ and the completed L-function
$L^*(u,\chi)$ as
\begin{equation}\label{frobenius}
  L^*(u,\chi) = \det(I-uq^{1/2}\Theta_\chi ), \quad \Theta_\chi =
  \mbox{diag}(e^{i\theta_1},\dots,e^{i\theta_N} )\;.
\end{equation}
The unitary matrix $\Theta_\chi$ (or rather, the conjugacy class of
unitary matrices) is called the unitarized Frobenius matrix of
$\chi$.

Taking the logarithmic derivative of \eqref{frobenius} we get an
Explicit Formula for primitive characters:
\begin{equation}\label{explicit formula for L}
  \Psi(n,\chi) =
-q^{n/2}\tr  \Theta_\chi^n - \lambda_\chi\;.
\end{equation}



\section{Prime polynomials in short intervals}\label{secGM}

In this section we prove Theorem~\ref{thm:GMff}, the analogue of the
Goldston-Montgomery result \eqref{MS result}.

\subsection{An involution} For $0\neq f\in \fq[T]$ we define
\begin{equation}
  f^*(T):=T^{\deg f} f(\frac 1T)
\end{equation}
or if  $f(T) = f_0+f_1T+\dots +f_n T^n$, $n=\deg f$ (so that
$f_n\neq 0$), then $f^*$ is the ``reversed'' polynomial
\begin{equation}
  f^*(T) = f_0T^n+f_1T^{n-1}+\dots +f_n\;.
\end{equation}
We also set $0^*=0$.

Note that  $f^*(0)\neq 0$ and $f(0)\neq 0$ if and only if $\deg f^*
= \deg f$. Moreover restricted to polynomials which do not vanish at
$0$, equivalently are co-prime to $T$, then $*$ is an involution:
\begin{equation}
f^{**} = f, \quad  f(0)\neq 0\;.
\end{equation}
We  also have multiplicativity:
\begin{equation}
  (fg)^* = f^* g^*\;.
\end{equation}

\begin{lemma}
For $f\in \mathcal P_n$ with $f(0)\neq 0$, we have $\Lambda(f^*) =
\Lambda(f)$.
\end{lemma}

\begin{proof}

For polynomials which do not vanish at $0$, i.e. are co-prime to
$T$, $P$ is irreducible if and only if $P^*$ is irreducible. This is
because if $P=AB$ with $A,B$ of positive degree then $P^*=(AB)^* =
A^* B^*$ and if $P(0)\neq 0$ then the same holds for $A,B$ and then
$\deg A^*=\deg A>0$, $\deg B^*=\deg B>0$ so $P$ is reducible;
applying $^*$ again and using that it is an involution (since
$P(0)\neq 0$) gives the reverse implication.
\end{proof}

\subsection{A fundamental relation}
We can now express the number of primes in our short intervals in
terms of the number of primes in a suitable arithmetic progression.
Define
\begin{equation}
 \tilde \Psi(n;Q,A) = \sum_{\substack{f\in \mathcal P_n\\ f=A\bmod Q}}
  \Lambda(f) \;.
\end{equation}
the sum over all polynomials of degree $n$, not necessarily monic.

\begin{lemma}
For $B\in \mathcal P_{n-h-1}$,
\begin{equation}\label{fund relation}
  \nu(T^{h+1}B;h) = \tilde \Psi(n;T^{n-h},B^*) \;.
\end{equation}
\end{lemma}
\begin{proof}
Let  $B\in \mathcal P_{n-h-1}$. We have $f= T^{h+1}B+g\in
I(T^{h+1}B;h)$, $g\in \mathcal P_{\leq h}$ if and only if $f^* = B^*
+T^{n-h}g^* $, and thus we find
\begin{equation}
  f\in I(T^{h+1}B;h) \Leftrightarrow f^* \equiv B^*  \mod T^{n-h}
  \;.
\end{equation}

As $f$ runs over $I(T^{h+1}B;h)$ with the proviso that $f(0)\neq 0$,
$f^*$ runs over all polynomials of degree exactly $n$ satisfying
$f^*\equiv B^* \mod T^{n-h}$, and for these
$\Lambda(f)=\Lambda(f^*)$.
\end{proof}


\subsection{Averaging}

We want to compute the mean value and variance of $\nu(A,h)$. To
perform the average over $A$, note that every monic polynomial $f\in
\mathcal M_n$ can be written uniquely as
\begin{equation}
  f=T^{h+1}B+g, \quad B\in \mathcal M_{n-(h+1)},\quad g\in  \mathcal P_{\leq
  h} \;.
\end{equation}
We therefore can decompose $\mathcal M_n$ as the disjoint union of
``intervals'' $I(T^{h+1}B;h)$ parameterized by $B\in \mathcal
M_{n-(h+1)}$:
\begin{equation}
  \mathcal M_n = \coprod_{B\in \mathcal M_{n-(h+1)}} I(T^{h+1}B ;h)
  \;.
\end{equation}

To compute averages $\nu$ on short intervals, it suffices, by the
foregoing, to take $A=T^{h+1}B$ and to average over all $B \in
\mathcal M_{n-(h+1)}$.

The map $*$ gives a bijection
\begin{equation}
  \begin{split}
  *: \mathcal M_{n-(h+1)} & \to
  \{  B^* \in \mathcal P_{\leq (n-h-1)}: B^*(0) =1 \}\\
     B &\mapsto  B^*
  \end{split}
\end{equation}
with polynomials of degree $\leq n-(h+1)$ with constant term $1$.
Thus as $B$ ranges over $\mathcal M_{n-(h+1)}$, $B^*$ ranges over
$(\fq[T]/(T^{n-h}))^\times$, all invertible residue class mod
$T^{n-h}$ so that $B^*(0)=1$.

Thus  the mean value is
\begin{equation}
\begin{split}
  \ave{\nu(\bullet;h)} &= \frac 1{\#\mathcal M_{n-h-1}} \sum_{B\in
    \mathcal M_{n-h-1}}
\nu(T^{h+1}B;,h) \\
&=\frac 1{q^{n-h-1}} \sum_{\substack{B^*\bmod T^{n-h}\\B^*(0)=1}}
\tilde \Psi (n;T^{n-h},B^*)
\end{split}
\end{equation}
and the variance is
\begin{equation}
\begin{split}
  \Var(\nu(\bullet;h)) &=\frac 1{\#\mathcal M_{n-h-1}} \sum_{B\in
    \mathcal M_{n-h-1}}
\left| \nu(T^{h+1}B;,h) - \ave{\nu} \right|^2 
\\
&=\frac 1{q^{n-h-1}} \sum_{\substack{B^*\bmod T^{n-h}\\B^*(0)=1}}
\left| \tilde \Psi (n;T^{n-h},B^*)- \ave{\nu} \right|^2 \;.
\end{split}
\end{equation}

\subsection{The mean value}
The computation of the mean value $\ave{\nu(\bullet;h)} = \frac
1{q^n}\sum_{A\in \mathcal M_n} \nu(A;h)$ is a simple consequence of
the Prime Polynomial Theorem. The result is
\begin{lemma}\label{lem:mean value}
Let $0<h<n$. The mean value of $\nu(A,;h)$ is
\begin{equation}
   \ave{\nu(\bullet;h)} =q^{h+1}(1-\frac{1}{q^n})\;.
\end{equation}
\end{lemma}
\begin{proof}
We do the computation in two different ways as a  check of the
all-important relation \eqref{fund relation}. By using the
definition of $\nu$, we get
\begin{equation}
  \begin{split}
 \ave{\nu(\bullet;h)}&= \frac 1{\#\mathcal M_{n-h-1}}
\sum_{B\in   \mathcal M_{n-h-1}}  \sum_{\substack{f\in
    I(T^{h+1}B;h)\\ f(0)\neq 0}} \Lambda(f) \\
&=   \frac 1{\#\mathcal M_{n-h-1}}\left( \sum_{f\in \mathcal M_n}
\Lambda(f) - \Lambda(T^n) \right) \;.
  \end{split}
\end{equation}
Note that
\begin{equation}
  \#\mathcal M_{n-h-1} = q^{n-h-1} = \frac {\Phi(T^{n-h})}{q-1}\;.
\end{equation}

Using \eqref{fund relation},  the mean value of $\nu(\bullet;h)$ is
\begin{equation}
  \begin{split}
  \ave{\nu(\bullet;h)} &= \frac 1{\Phi(T^{n-h})} \sum_{\substack{ B^*
      \bmod T^{n-h} \\ B^*(0)=1}} \tilde \Psi(n;T^{n-h},B^*)\\
&= \frac 1{\Phi(T^{n-h})} \sum_{\substack{\deg
    f^*=n\\ f^*(0)=1}} \Lambda(f^*) \\
&= \frac 1{\Phi(T^{n-h})}  \left( \sum_{\deg f^*=n} \Lambda(f^*) -
\sum_{c\in \fq^*} \Lambda(cT^n)\right)\\
&= \frac 1{q^{n-h-1}}\left( \sum_{f^*\in \mathcal M_n} \Lambda(f^*)
-\Lambda(T^n) \right)\;.
   \end{split}
\end{equation}
Hence
\begin{equation}
  \ave{\nu(\bullet;h)} = \frac1{q^{n-h-1}} \left(q^n-1
  \right) =q^{h+1}(1-\frac{1}{q^n})
\end{equation}
on using the Prime Polynomial Theorem in the form \eqref{Explicit
formula2}.
\end{proof}

\subsection{An alternate expression for $\nu(A;h)$}

Using the standard orthogonality relation \eqref{standard
orthogonality} for Dirichlet characters modulo $T^{n-h}$ gives an
alternate expression for $\tilde \Psi(n;T^{n-h},B^*)$ and hence for
$\nu(T^{h+1}B;h)$:
\begin{equation}
   \tilde
\Psi(n;T^{n-h},B^*) = \frac 1{\Phi(T^{n-h})} \sum_{\chi \bmod
  T^{n-h}} \overline{\chi}(B^*) \sum_{\deg f^*=n}
\Lambda(f^*)\chi(f^*)\;.
\end{equation}

Only {\em even} characters give a non-zero term, because
$\Lambda(cf) = \Lambda(f)$ for $c\in \fq^\times$, and each even
character contributes a term
\begin{equation}
  \overline{\chi}(B^*) \frac{q-1}{\Phi(T^{n-h})} \sum_{\substack{\deg f=n\\
\mbox{monic}}} \Lambda(f)\chi(f) =  \overline{\chi}(B^*) \frac
1{q^{n-h-1}} \Psi(n,\chi)
\end{equation}
where
\begin{equation}
  \Psi(n,\chi) =  \sum_{\substack{\deg f=n\\ \mbox{ monic}}}
  \Lambda(f)\chi(f)\;.
\end{equation}
Note that the number of even characters mod $T^{n-h}$ is exactly
$\frac 1{q-1}\Phi(T^{n-h}) = q^{n-h-1}$.

The trivial character $\chi_0$ contributes the term
\begin{equation}
  \frac{(q-1)(q^n-1)}{\Phi(T^{n-h})} = q^{h+1}(1-\frac{1}{q^{n}}) =
  \ave{\nu}\;.
\end{equation}
Thus we find that the difference between $\nu(T^{h+1}B;h)$ and its
mean $\ave{\nu}$ is
\begin{equation}\label{form of nu}
\nu(T^{h+1}B;h)-\ave{\nu} =
\frac 1{q^{n-h-1}}\sum_{\substack{\chi \neq \chi_0 \bmod T^{n-h}\\
\mbox{ even}}} \overline{\chi}(B^*)  \Psi(n,\chi)\;.
\end{equation}

%
%
%
%
%
%
%
%
\subsection{The variance}

Our result here is that
\begin{theorem}
Fix $n>0$ and let $0<h<n$. As $q\to \infty$, the variance of $\nu$
is given by
\begin{equation}
  \Var(\nu) = q^{h+1} \cdot \left( \frac 1{q^{n-h-1}} \sum^*_{\chi}|\tr
  \Theta_\chi^n|^2 + O(\frac{n-h}{q^{n/2}} + \frac{n^2}{q})  \right)
\end{equation}
where the sum is over primitive even characters modulo $T^{n-h}$,
the implied constant depending only on $n$.
\end{theorem}

 \begin{proof}
By \eqref{form of nu} we have
\begin{equation}
  \Var(\nu) = \frac 1{q^{n-h-1}}
\sum_{\substack{B^* \bmod T^{n-h} \\B^*(0) =1 }} \frac
1{q^{2(n-h-1)} }\left| \sum_{\substack{\chi\neq \chi_0\\ \mbox{even}
}}  \overline{\chi}(B^*) \Psi(n,\chi)\right|^2 \;.
\end{equation}
Expanding the sum over characters, and interchanging the order of
summation to use  the orthogonality relation of Lemma~\ref{special
orthogonality} gives
\begin{equation}
  \Var(\nu) = \frac 1{q^{2(n-h-1)} }\sum_{\substack{\chi\neq
      \chi_0\\\mbox{even}}} |\Psi(n,\chi)|^2     \;.
\end{equation}

There are altogether $\varphi(T^{n-h})/(q-1) = q^{n-h-1}$ even
characters modulo $T^{n-h}$, of which $O(q^{n-h-2})$ are
non-primitive. We bound the contribution of the nontrivial
non-primitive characters by $\Psi(n,\chi) = O(nq^{n/2})$ via the
Riemann Hypothesis.  Thus the non-primitive  characters contribute a
total of $O( n^2 q^h )$ to $\Var(\nu)$.

Using the Explicit Formula \eqref{explicit formula for L} for
primitive even characters and the Riemann Hypothesis  gives
\begin{equation}
|\Psi(n,\chi) |^2 =q^n |\tr \Theta_\chi^n|^2 +O((n-h) q^{n/2})\;.
\end{equation}
Therefore
\begin{equation}
  \Var(\nu) = q^{h+1} \cdot \left( \frac 1{q^{n-h-1}} \sum^*_{\chi}|\tr
  \Theta_\chi^n|^2 + O(\frac{n-h}{q^{n/2}} + \frac{n^2}{q})  \right)
\end{equation}
where the sum is over primitive even characters modulo $T^{n-h}$,
whose number is $q^{n-h-1}(1-\frac 1q)$.
\end{proof}

\subsection{Proof of Theorem~\ref{thm:GMff}}

Thus we found that for $h<n-3$, the variance of $\nu$ is given by
\begin{equation}
  \frac 1{q^{h+1}}\Var(\nu) = (1-\frac 1q) \ave{|\tr
\Theta_\chi^n|^2} + O(\frac{ n-h}{q^{n/2}} + \frac{n^2}{q})
\end{equation}
with $\ave{|\tr \Theta_\chi^n|^2}$ being the mean value of $|\tr
\Theta_\chi^n|^2$ over the set of all  primitive even Dirichlet
characters modulo $T^{n-h}$.
Thus as $q\to \infty$, $\Var(\nu)/q^{h+1}$ is asymptotically equal
to the "form factor" $\ave{|\tr \Theta_\chi^n|^2}$.

 To proceed further, we need to invoke a recent result of N.~Katz
 \cite{KatzKR2}:
\begin{theorem}\cite[Theorem 1.2]{KatzKR2} \label{thm:KatzKR2}
Fix\footnote{If the characteristic of $\fq$ is different than $2$ or
  $5$ then the result also holds for $m=2$.} $m\geq 3$. The unitarized
Frobenii $\Theta_\chi$ for the family of even primitive characters
mod $T^{m+1}$ become equidistributed in the projective unitary group
$PU(m-1)$ of size $m-1$, as $q\to \infty$.
\end{theorem}



Applying Theorem~\ref{thm:KatzKR2} gives
\begin{equation}
\lim_{q\to \infty}  \frac 1{q^{n-h-1}(1-\frac 1q)} \sum^*_{\chi}|\tr
\Theta_\chi^n|^2  = \int_{PU(n-h-2)} |\tr U^n|^2 dU\;.
\end{equation}
We may pass from the projective unitary group $PU(n-h-2)$ to the
unitary group because the function $|\tr U^n|^2$ being averaged is
invariant under scalar multiplication. As is well known (see e.g.
\cite{DS}), for $n>0$
\begin{equation}
\int_{U(N)} |\tr U^n|^2 dU = \min(n,N)\;.
\end{equation}
Therefore we find
\begin{equation}
  \Var(\nu) \sim q^{h+1}(n-h-2),\quad q\to \infty\;.
\end{equation}
This concludes the proof of Theorem~\ref{thm:GMff}.

\section{Prime polynomials in arithmetic progressions}\label{secFGH}

In this section we prove Theorem~\ref{main thm}, giving the function
field analogue of the conjectures of Hooley \eqref{Hooley conj} and
Friedlander-Goldston \eqref{FG conj}.

\subsection{The  range $n<\deg Q$}
\label{sec:uninteresting}

We prove the result in the  range $n<\deg Q$ by elementary
arguments:

\begin{proposition}
  For $0<n<\deg Q$, we have
  \begin{equation}
    G(n;Q) = nq^n -\frac{q^{2n}}{\Phi(Q)} + O(n^2 q^{n/2}) +O((\deg Q)^2)
  \end{equation}
where the implied constant is independent of $q$, $n$ and $Q$.
\end{proposition}
\begin{proof}
  Assume as we may that $\deg A<\deg Q$. If $n<\deg Q$ then the only
  solution to the congruence $N=A\bmod Q$, with $\deg N=n<\deg Q$ is
  $A$ (if $\deg A=n$) or else there is no solution. Therefore if
  $n<\deg Q$ then
\begin{equation}
    \Psi(n;Q,A) =
    \begin{cases}
      \Lambda(A),& A \mbox{ is monic and }\deg A=n \\0,& \mbox{ otherwise.}
    \end{cases}
\end{equation}
Thus
 \begin{equation*}
   \begin{split}
     G(n;Q) &= \sum_{\gcd(A,Q)=1} \left| \frac{q^n}{\Phi(Q)}-
     \begin{cases}
       \Lambda(A),& A \mbox{ is monic and}\deg A=n \\0,& \mbox{ otherwise}
     \end{cases} \right|^2 \\
 &= \sum_{\substack{\deg A=n\\A\mbox{ monic}\\ \gcd(A,Q)=1}}  \Lambda(A)^2 -
 2\frac{q^n}{\Phi(Q)} \sum_{\substack{\deg A=n\\A\mbox{ monic}\\ \gcd(A,Q)=1}}
 \Lambda(A) + \frac{q^{2n}}{\Phi(Q)}\;.
   \end{split}
 \end{equation*}

By the Prime Polynomial Theorem \eqref{Explicit formula2},
\begin{equation}
  \sum_{\substack{\deg A=n\\A\mbox{ monic}\\ \gcd(A,Q)=1}} \Lambda(A)
  = q^n-\sum_{\substack{P\mid Q \mbox{ prime}\\ \deg P\mid n}} \deg P
  = q^n +O(\deg Q)\;.
\end{equation}
According to Lemma~\ref{variants of PNT},
\begin{equation}
  \begin{split}
   \sum_{\substack{\deg A=n\\A\mbox{ monic}\\ \gcd(A,Q)=1}}
   \Lambda(A)^2 & =
\sum_{\deg A=n}\Lambda(A)^2 - \sum_{\substack{P\mid Q\\\deg P\mid
n}} (\deg P)^2
\\&=nq^n +O(n^2q^{n/2}) + O( (\deg Q)^2)
  \end{split}
\end{equation}
 and so we find
\begin{equation}
  G(n;Q) = nq^n - \frac{q^{2n}}{\Phi(Q)} + O(n^2q^{n/2}) +O((\deg Q)^2) + O(\frac{q^n}{\Phi(Q)} \deg Q) \;.
\end{equation}
Since for $n<\deg Q$,
\begin{equation}
\frac{q^n}{\Phi(Q)} \leq \frac 1q \prod_{\substack{P\mid Q\\
\mbox{prime}}} (1-\frac 1{|P|})^{-1} \leq \frac 1q
\prod_{\substack{\deg P\leq \deg Q\\ \mbox{prime}}} (1-\frac
1{|P|})^{-1}  \ll \frac{\deg Q}{q}
\end{equation}
we find
\begin{equation}
  G(n;Q) = nq^n - \frac{q^{2n}}{\Phi(Q)} + O(n^2q^{n/2}) +O((\deg Q)^2)
\end{equation}
as claimed.
\end{proof}

\subsection{The range $n\geq \deg Q$}

To deal with the  range $n\geq \deg Q$ we relate the problem to an
equidistribution statement for the unitarized Frobenii of primitive
odd characters.   It transpires that   $G(n;Q)$ is related to the
mean value of the modulus squared of the trace of the Frobenius
matrices associated to the family of Dirichlet L-functions for
characters modulo $Q$:
\begin{theorem}
Fix $n$  and let $Q\in \fq[T]$ have degree $\deg Q\geq 2$. Then
\begin{equation}\label{fund identity}
  \frac {G(n;Q)}{q^n}=
\ave{  |\tr \Theta_\chi^n|^2} (1+ \frac 1q) +
  O(\frac{(\deg Q)^2}q)
\end{equation}
where $\ave{}$ denotes the average over all odd primitive characters
modulo $Q$.
\end{theorem}

\begin{proof}
The orthogonality relation \eqref{orthogonality relation 1}  gives
\begin{equation}
\begin{split}
  \Psi(n;Q,A) &= \frac 1{\Phi(Q)} \sum_{\chi \bmod Q} \bar \chi(A)
  \sum_{\deg N=n} \chi(N)\Lambda(N)\\
&= \frac 1{\Phi(Q)} \sum_{\chi \bmod Q} \bar \chi(A) \Psi(n,\chi)\;.
\end{split}
\end{equation}
The trivial character $\chi_0$ gives a contribution of
\begin{equation}
\frac 1{\Phi(Q)}   \sum_{\substack{\deg N=n\\ \gcd(N,Q)=1}}
\Lambda(N) = \frac {q^n}{\Phi(Q)}  - \frac
1{\Phi(Q)}\sum_{\substack{P\mid
    Q\\ \deg P\mid n}} \deg P\;.
\end{equation}
Hence
\begin{equation}
  \Psi(n;Q,A)-\frac {q^n}{\Phi(Q)} =  - \frac 1{\Phi(Q)}\sum_{\substack{P\mid
    Q\\ \deg P\mid n}} \deg P + \frac 1{\Phi(Q)}\sum_{\chi\neq \chi_0}
  \chi(A) \Psi(n,\chi)\;.
\end{equation}

We square out and average over all $A\bmod Q$ coprime with $Q$.
Using the orthogonality relation \eqref{OR2} gives
\begin{equation}
  G(n;Q) = \frac 1{\Phi(Q)} \sum_{\chi\neq \chi_0} |\Psi(n,\chi)|^2 +
   \frac 1{\Phi(Q)}(\sum_{\substack{P\mid Q\\ \deg P\mid n}} \deg P
   )^2\;.
\end{equation}

For nontrivial characters which are either even or imprimitive, we
use the Riemann Hypothesis \eqref{RHWeil} to bound $|\Psi(n,\chi)|^2
\leq q^n(\deg Q-1)^2$. Therefore we find
\begin{multline}
   G(n;Q) = \frac 1{\Phi(Q)} \sum_{\chi \mbox { primitive, odd}}
     |\Psi(n,\chi)|^2 \\
+ O( q^n(\deg Q)^2 \frac {\#\{\chi \mbox{ either
         even or imprimitive\}}}{\Phi(Q)})\;.
\end{multline}

The number of even characters is $\Phi(Q)/(q-1)$, and the number of
imprimitive characters is $O(\Phi(Q)/q)$. Hence the remainder term
above is bounded by  $O( q^{n-1}(\deg Q)^2)$.

For each primitive odd character, the ``explicit formula''
\eqref{explicit formula for L} says
\begin{equation}
  \Psi(n,\chi) = -q^{n/2}\tr \Theta_\chi^n 
\end{equation}
and therefore
\begin{equation}
  G(n;Q) = q^n\frac 1{\Phi(Q)} \sum_{ \chi \mbox{ odd primitive} } \left|\tr
  \Theta_\chi^n \right|^2 +O(q^{n-1} (\deg Q)^2)\;.
\end{equation}
Replacing $\Phi(Q) $ by the number of odd primitive characters times
$1+O(\frac 1q)$ gives \eqref{fund identity}.
\end{proof}


We now use another recent equidistribution  result of  Katz
\cite{KatzKR}:
\begin{theorem}[Katz \cite{KatzKR}]\label{thm:KatzKR}
Fix $m\geq 2$. Suppose we are given a sequence of finite fields
$\fq$ and squarefree polynomials $Q(T)\in \fq[T]$ of degree $m$. As
$q\to \infty$, the conjugacy classes $\Theta_{\chi}$ with $\chi$
running over all primitive odd characters modulo $Q$, are uniformly
distributed in the unitary group $U(m-1)$.
\end{theorem}

Using Theorem~\ref{thm:KatzKR} we get for $n>0$,
\begin{equation}
\lim_{q\to \infty} \ave{ \left|\tr \Theta_\chi^n\right|^2 }
 =\int_{U(\deg Q-1)} \left|\tr U^n \right|^2 dU
\end{equation}
where $dU$ is the Haar probability measure on the unitary group
$U(N)$. Since \cite{DS}
\begin{equation}
 \int_{U(N)} \left|\tr U^n \right|^2 dU =\min(n,N)\;,
\end{equation}
we find
\begin{equation}
\lim_{q\to \infty} \frac{ G(n;Q) }{q^n} = \min(n,\deg Q-1)
\end{equation}
which is the statement of  Theorem~\ref{main thm}.

\appendix

\section{A calculation based on a Hardy-Littlewood-type conjecture}

In the number-field setting, the problems we have considered here
have previously been explored using the Hardy-Littlewood conjecture
relating to the density of generalized twin primes \cite{FG, MS}.
In this appendix we sketch a heuristic calculation showing how the
corresponding conjecture in the function-field setting may be used
in the same way.  As an example, we focus on estimating $G(n, Q)$.

The twin prime conjecture of Hardy and Littlewood for the rational
function field $\fq[T]$ states that, given a polynomial $0\neq K\in
\fq[T]$, and $n>\deg K$,
\begin{equation}\label{HL conj}
  \sum_{{\rm deg}f=n}\Lambda(f)\Lambda(f+K) \sim \mathfrak S(K)
  q^n
\end{equation}
as $q^n \to \infty$, where the "singular series" $\mathfrak S(K)$ is
given by
\begin{equation}\label{Sdef}
 \mathfrak S(K) = \prod_{P} (1-\frac 1{|P|})^{-2} (1-\frac {\nu_K(P)}{|P|}),
\end{equation}
with the product involving all monic irreducible $P$ and
\begin{equation}
  \nu_K(P) = \#\{A\mod P: A(A+K)=0\bmod P\}=
  \begin{cases}
    1,& P\mid K\\ 2,& P\nmid K.
  \end{cases}
\end{equation}
While for fixed $q$ and $n\to \infty$ the problem is currently
completely open, for fixed $n$ and $q\to \infty$, \eqref{HL conj} is
known to hold \cite{Bender Pollack, BS} for $q$ odd, in the form
\begin{equation}\label{BP thm}
\sum_{{\rm deg}f=n}\Lambda(f)\Lambda(f+K) =  q^n +O_n(q^{n-\frac
12}) \;.
\end{equation}
Note that $\mathfrak S(K) = 1+O_n(\frac 1q)$.

 We want to use \eqref{HL conj} to compute $G(n;Q)$ and to
show that the result is consistent with
\begin{equation}
  G(n;Q) \sim q^n (\deg Q-1), \quad n> \deg Q\;.
\end{equation}
It turns out that this can be done if we ignore the contribution
from the remainder implicit in \eqref{HL conj}. The remainder term
in \eqref{BP thm} is insufficient for our purposes.

Starting with
\begin{equation}
  G(n;Q) = \sum_{\gcd(A,Q)=1} \left| \Psi(n;Q,a)-\frac{q^n}{\Phi(Q)} \right|^2\;,
\end{equation}
we have
\begin{equation}\label{expanding}
  G(n;Q) =  \sum_{\gcd(A,Q)=1}  \Psi(n;Q,A)^2 - 2\frac{q^n}{\Phi(Q)}
  \sum_{\gcd(A,Q)=1}  \Psi(n;Q,A)  + \frac{q^{2n}}{\Phi(Q)}\;.
\end{equation}

The first moment of $\Psi(n;Q,A)$ is
\begin{equation}
  \begin{split}
   \sum_{\gcd(A,Q)=1}  \Psi(n;Q,A) &= \sum_{\substack{\deg
       f=n\\\gcd(f,Q)=1}} \Lambda(f) \\
& = \sum_{\deg f=n} \Lambda(f) - \sum_{\substack{\deg f=n\\\deg
    \gcd(f,Q)>0}} \Lambda(f)\\
& = q^n - \sum_{\substack{\deg P\mid n\\P\mid Q\mbox{ prime}}} \deg
P\;.
  \end{split}
\end{equation}
By Lemma~\ref{variants of PNT} we may safely replace
\begin{equation}
   \sum_{\gcd(A,Q)=1}  \Psi(n;Q,A) = q^n + \mbox{ negligible.}
\end{equation}

For the second moment of $\Psi(n;Q,A)$ we have
\begin{equation}
  \begin{split}
   \sum_{\gcd(A,Q)=1}  \Psi(n;Q,A)^2 & =
\sum_{\substack{\deg f=\deg g =n\\ f\equiv g \mod Q\\\gcd(f,Q)=1}}
\Lambda(f)\Lambda(g) \\
&= \sum_{\substack{\deg f=n\\ \gcd(f,Q)=1}} \Lambda(f)^2 +
\sum_{\substack{\deg f=\deg g
       =n\\ f\equiv g \mod Q \\f\neq g\\ \gcd(f,Q)=1}}
       \Lambda(f)\Lambda(g)\;.
  \end{split}
\end{equation}
Now
\begin{equation}
   \sum_{\substack{\deg f=n\\ \gcd(f,Q)=1}} \Lambda(f)^2 = nq^n
     +O(n^2q^{n/2}) -\sum_{\substack{P\mid Q\\\deg P\mid n}} (\deg
     P)^2\;.
\end{equation}
For the sum over $f\neq g$, we write the condition $f=g\mod Q$ as
$g=f+JQ$, $J\neq 0$, $\deg J<n-\deg Q$ (the number of such $J$ of
degree $j$ is $(q-1)q^{j}$) and then
\begin{equation}\label{congruent sum}
\sum_{\substack{\deg f=\deg g =n\\ f\equiv g \mod Q \\f\neq g\\
\gcd(f,Q)=1}} \Lambda(f)\Lambda(g) = \sum_{\substack{\deg J<n-\deg
Q\\ J\neq 0}} \psi_2(n;JQ)
\end{equation}
where for $K\neq 0$, $\deg K<n$,
\begin{equation}
   \psi_2(n;K):=\sum_{\substack{\deg f=n\\ f {\rm monic}}} \Lambda(f)\Lambda(f+K)\;.
\end{equation}

Clearly we can split the right hand side of \eqref {congruent sum}
as follows
\begin{equation}
\sum_{\substack{\deg f=\deg g =n\\ f\equiv g \mod Q \\f\neq g\\
\gcd(f,Q)=1}} \Lambda(f)\Lambda(g) =\sum_{j=0}^{n-\deg Q}
\sum_{\substack{\deg J=j\\ J\neq 0}} \psi_2(n;JQ)\;.
\end{equation}
The $J$-sum here is not restricted to monic polynomials.  We can
restrict it to monics, multiplying by $q-1$.  Then inserting
\eqref{HL conj} we have
\begin{equation}\label{congruent sum1}
\sum_{\substack{\deg f=\deg g =n\\ f\equiv g \mod Q \\f\neq g\\
\gcd(f,Q)=1}} \Lambda(f)\Lambda(g) \sim q^n(q-1)\sum_{j=0}^{n-\deg
Q} \sum_{\substack{\deg J=j\\ J\neq 0\\J {\rm monic}}}\mathfrak
S(JQ)
\end{equation}
as $q^n\rightarrow\infty$.

In order to estimate the $J$-sum in \eqref{congruent sum1}, consider
\begin{equation}
\sum_{J {\rm monic}}\frac{\mathfrak S(JQ)}{|J|^s}=\alpha \sum_{J
{\rm monic}}\frac{1}{|J|^s}\prod_{P|JQ}\frac{|P|-1}{|P|-2}
\end{equation}
where the equality follows from inserting \eqref{Sdef} and
\begin{equation}
\alpha=\prod_P\left(1-\frac{1}{(|P|-1)^2}\right)\;.
\end{equation}
Hence
\begin{equation}
\sum_{J {\rm monic}}\frac{\mathfrak
S(JQ)}{|J|^s}=\alpha\prod_{P|Q}\frac{|P|-1}{|P|-2} \sum_{J {\rm
monic}}\frac{1}{|J|^s}\prod_{\substack {P|J \\ P\nmid
Q}}\frac{|P|-1}{|P|-2}\;.
\end{equation}
Since the summand on the right hand side is multiplicative, we may
write this as
\begin{equation}
\sum_{J {\rm monic}}\frac{\mathfrak
S(JQ)}{|J|^s}=\alpha\prod_{P|Q}\frac{|P|-1}{|P|-2} \prod_{P\nmid
Q}\left(1+\frac{1}{|P|^s-1}\frac{|P|-1}{|P|-2}\right)
\prod_{P|Q}\left(1-\frac{1}{|P|^s}\right)^{-1}\;.
\end{equation}
Therefore
\begin{equation}
\sum_{J {\rm monic}}\frac{\mathfrak
S(JQ)}{|J|^s}=\alpha\zeta_A(s)\prod_{P|Q}\frac{|P|-1}{|P|-2}
\prod_{P\nmid Q}\left(1+\frac{1}{|P|^s(|P|-2)}\right)
\end{equation}
with
\begin{equation}
\zeta_A(s)=\prod_P\left(1-\frac{1}{|P|^s}\right)^{-1}.
\end{equation}
Hence
\begin{equation}
\sum_{J {\rm monic}}\frac{\mathfrak
S(JQ)}{|J|^s}=\alpha\zeta_A(s)\prod_{P|Q}\frac{|P|-1}{|P|-2}\left(1+\frac{1}{|P|^s(|P|-2)}\right)^{-1}
\prod_{P}\left(1+\frac{1}{|P|^s(|P|-2)}\right)\;.
\end{equation}
Furthermore
\begin{multline}
\sum_{J {\rm monic}}\frac{\mathfrak S(JQ)}{|J|^s}=\alpha\zeta_A(s)\zeta_A(s+1)\prod_{P|Q}\frac{|P|-1}{|P|-2}\left(1+\frac{1}{|P|^s(|P|-2)}\right)^{-1}\\
\times
\prod_{P}\left(1+\frac{2}{|P|^{s+1}(|P|-2)}-\frac{|P|}{|P|-2}\frac{1}{|P|^{2s+2}}\right)\;.
\end{multline}

It is convenient to re-express these formulae in terms of the
variable $u=1/q^s$.  Thus $|J|=u^{-{\rm deg}J}$, $|P|=u^{-{\rm
deg}P}$, and
\begin{multline}\label{uform}
\sum_{J {\rm monic}}\mathfrak S(JQ)u^{{\rm deg}J}=\alpha Z(u)Z(u/q)\prod_{P|Q}\frac{|P|-1}{|P|-2}\left(1+\frac{u^{{\rm deg}P}}{(|P|-2)}\right)^{-1}\\
\times \prod_{P}\left(1+\frac{2u^{{\rm
deg}P}}{|P|(|P|-2)}-\frac{u^{2{\rm deg}P}}{|P|(|P|-2)}\right)
\end{multline}
with
\begin{equation}
Z(u)=\prod_P(1-u^{{\rm deg}P})^{-1}=\frac{1}{1-qu}\;.
\end{equation}

We can now estimate the $J$-sum in \eqref{congruent sum1} by
denoting
\begin{equation}
F(u)=\sum_{J {\rm monic}}\mathfrak S(JQ)u^{{\rm deg}J}
\end{equation}
and using
\begin{equation}
\sum_{\substack{\deg J=j\\ J\neq 0\\J {\rm monic}}}\mathfrak
S(JQ)=\frac{1}{2\pi{\rm i}}\oint \frac{F(u)}{u^{j+1}}{\rm d}u\;,
\end{equation}
where the contour is a small circle enclosing the origin but no
other singularities of the integrand.  Expanding the contour beyond
the poles of $F(u)$ at $u=1/q$ and $u=1$ (coming from the factors of
$Z(u)$ and $Z(u/q)$ in \eqref{uform}), we find that as
$q\rightarrow\infty$
\begin{equation}\label{H-Lav}
\sum_{\substack{\deg J=j\\ J\neq 0\\J {\rm monic}}}\mathfrak
S(JQ)\sim q^j\frac{|Q|}{\Phi(Q)}-\frac{1}{q-1},
\end{equation}
where we have used
\begin{equation}
\prod_{P|Q}\frac{|P|}{|P|-1}=\frac{|Q|}{\Phi(Q)}.
\end{equation}
Note that the first term in \eqref{H-Lav} coincides after the usual
translation with that in the corresponding expression in the number
field calculation \cite{FG}, but that interestingly the second term
has a different form.

 Finally, substituting \eqref{H-Lav} into
\eqref{congruent sum1} and incorporating the estimates for the other
terms in \eqref{expanding}, we find that
\begin{equation}
G(n;Q)\sim q^n\left({\rm deg}Q-\frac{|Q|}{\Phi(Q)}\right).
\end{equation}

We now observe that as $q \rightarrow \infty$
\begin{equation}
\frac{|Q|}{\Phi(Q)}\rightarrow 1
\end{equation}
and so in this limit, when $n$ is fixed with ${\rm deg} Q \le n+1$,
this calculation matches Theorem~\ref{main thm}.  Furthermore, when
${\rm deg} Q \rightarrow \infty$ with $q$ fixed we have that
\begin{equation}
G(n;Q)\sim q^n {\rm deg}Q
\end{equation}
which is consistent with the Hooley's conjecture~\eqref{Hooley conj}
in the number field case.

\bigskip

\noindent{\bf Acknowledgements:} We thank Nick Katz for several
discussions, and  Julio Andrade  and the referees for their
comments.

\end{document}